\documentclass{article}

\usepackage{amssymb, amsmath, mathabx, amsthm}

\usepackage{graphicx}
\usepackage{enumerate}

\usepackage[utf8]{inputenc}
\usepackage{csquotes, subcaption, xspace}
\usepackage[capitalise,noabbrev]{cleveref}

\newtheorem{theorem}{Theorem}[section]
\newtheorem{lem}[theorem]{Lemma}
\newtheorem{conj}[theorem]{Conjecture}

\theoremstyle{definition}

\theoremstyle{remark}

\numberwithin{equation}{section}
\newcommand{\ds}{\displaystyle}

\newcommand{\pochhammer}[2]{\left(#1\right)_{#2}}

\usepackage{cancel}

\crefname{equation}{}{Equations}

	\title{Taylor series for generalized Lambert W functions}
	\author{Paul Castle}
	
\begin{document}

		\maketitle
		
		\begin{abstract}
			The Lambert W function gives the solutions of a simple exponential polynomial. The generalized Lambert W function was defined in \cite{mezHo2017generalization} and has found applications in delay differential equations and physics. In this article we describe an even more general function, the inverse of a product of powers of linear functions and one exponential term. We show that the coefficients of the Taylor series for these functions can be described by multivariable hypergeometric functions of the parameters. We also present a surprising conjecture for the radius of convergence of the Taylor series, with a rough proof. 
		\end{abstract}
		
		\section{The generalized Lambert W function}
		
		The Lambert W function is the multivalued inverse of the function $z e^z$, studied by Lambert and Euler \cite{Corless1996LambertW}. The broadest generalization of this function was first described by Scott et al. in \cite{scott2006general} and analysed further by Mez\"{o} and Baricz in \cite{mezHo2017generalization}. These functions have applications in delay differential equations, eigenvalue problems in quantum mechanics, the $n$-body problem in general relativity, and investigations of Bose-Fermi mixtures \cite{mezHo2016some}. In \cite{mezHo2017generalization}, the generalized Lambert W function 
		$$W\left( \begin{array}{l l l l} 
		t_1	&	\ldots	&	t_n\\
		s_1	&	\ldots	&	s_m
		\end{array}; x \right)$$
		is defined as the (generally multivalued) inverse of the function 
		$$\frac{(x - t_1)(x - t_2) \ldots (x - t_n)}{(x - s_1)(x - s_2)\ldots (x-s_m)} e^x.$$
		Mez\"{o} and Baricz in \cite{mezHo2017generalization} calculate the Taylor series for two cases:
		$W\left(\begin{array}{l} 
		t\\	s
		\end{array}; a \right)$, with the coefficients in terms of Laguerre polynomials, and $W\left(\begin{array}{l l} 
		t_1 & t_2\\	-
		\end{array}; a \right)$, with the coefficients in terms of Bessel polynomials. Both Laguerre and Bessel polynomials can be defined in terms of the generalized hypergeometric function ${}_2F_0$. They also find the radius of convergence for the first function. Calculating the Taylor series for the general case, and its radius of convergence, is proposed as a further research direction in \cite{mezHo2017generalization}.

		In this paper we investigate the Taylor series of an even more general function. Consider the complex function $f(z)$ defined by the principle value of 
		$$z (z - t_1)^{p_1}(z - t_2)^{p_2} \ldots (z - t_n)^{p_m} e^z,$$
		where the $t_i$ and $p_i$ can be any nonzero complex numbers. The extra $z$ is used for convenience when taking the Taylor series. We denote the inverse of this function by 
		$$W^{(p)}\left( \begin{array}{l l l l} 
		t_1	&	\ldots	&	t_m\\
		p_1	&	\ldots	&	p_m
		\end{array}; x \right).$$
		Note that $f(z)$ may have very complicated partial inverses in the complex plane if the $p_i$ are irrational. However the only value of $z$ for which $f(z) = 0$ is $z = 0$. If the extra $z$ is replaced with the more general $(z - t_0)^{p_0}$, the function can be put into this form using the substitution $w = \frac{z - t_0}{p_0}$. Then we have:
		\begin{align*}
		y	&= (z - t_0)^{p_0}(z - t_1)^{p_1}(z - t_2)^{p_2} \ldots (z - t_n)^{p_m} e^z\\
		\frac{1}{p_0} y^{1/p_0}	&= \frac{1}{p_0} z' (z' + t_0 - t_1)^{\frac{p_1}{p_0}} \ldots (z' + t_0 - t_n)^{\frac{p_m}{p_0}} e^{\frac{z'}{p_0} + \frac{t_0}{p_0}},\\
		p_0^{\frac{1}{p_0} \sum p_i } e^{-  \frac{t_0}{p_0}}y^{1/p_0}	&= w \left(w - \frac{t_1 - t_0}{p_0} \right)^{\frac{p_1}{p_0}} \ldots \left(w - \frac{t_m - t_0}{p_0} \right)^{\frac{p_m}{p_0}} e^{\frac{w}{p_0}},
		\end{align*}
		which can be solved using the function $W^{(p)}$. If the $p_i$ are all $1$ or $-1$ then $f(x)$ can be expressed as a rational function multiplied by $e^x$. So $W^{(p)}$ is strictly more general than $W$.  
		
		\begin{theorem}
			Let $f(z) = z (z - t_1)^{p_1} \ldots (z - t_m)^{p_m} e^z$. Then the Taylor series for the inverse of $f$ around $0$ is 
			\begin{equation} \label{series}
			f^{-1}(z) = \sum_{n=1}^\infty \frac{(-n)^{n-1}}{n!} (-t_1)^{-n p_1} \ldots (-t_m)^{-n p_m} F_n z^n,
			\end{equation}
			where
			\begin{align*}
			F_n &= {}^{1+1}F_{0+0} \left( \begin{array}{l l l}
			1 - n	&; n p_1, \ldots, n p_m	&; \\
			&;						&;
			\end{array} \frac{1}{n t_1}, \ldots, \frac{1}{n t_m} \right)\\
			&= \sum_{k_1, \ldots, k_m =0}^\infty \frac{(1-n)_{k_1 + \ldots + k_m} (np_1)_{k_1} \ldots (np_m)_{k_m}}{k_1! \ldots k_m! (nt_1)^{k_1} \ldots (nt_m)^{k_m}}.
			\end{align*}
			is a generalized Kamp\'{e} de F\'{e}riet function, a limiting case of the Lauricella function. This function always reduces to a polynomial of degree $n$. 
		\end{theorem}
		
		We prove this theorem in \cref{Main theorem section}. We also make some progress in finding the radius of convergence of the Taylor series about $z = 0$. We have the following conjecture:
		
		\begin{conj} \label{Radius conjecture}
			There exist integers $l_1, \ldots, l_m$ such that the radius of convergence for the series \cref{series} is
			\begin{equation} 
			R = \left| \exp\left(2\pi i \sum_j  l_j \phi_j  - 1\right) \left(1 - \sum_j \phi_j\right)^{1 - \sum_j \phi_j} \prod_j \left(\frac{t_j}{p_j + \phi_j} \right)^{p_j + \phi_j} p_j^{p_j} \phi_j^{\phi_j} \right|,
			\end{equation}			
			where $\phi$ is one of the solutions to the system of quadratic equations
			\begin{equation}
			\left(1 - \sum_j \lambda_j\right)(p_i + \lambda_i) + \lambda_i t_i = 0.
			\end{equation}
		\end{conj}
		Using an asymptotic for $F_n$ and the saddle point method, we give a strong argument for this conjecture in \cref{ROC section}, but not a full proof.

%

	\section{Multi-valued hypergeometric functions}
	%
	%
	%
	%
	
	This paper uses polynomials that can be expressed as multivariable hypergeometric functions \cite{exton1976multiple}. We only give the definitions here.
	
	\subsection{Lauricella functions}
	The four Lauricella hypergeometric series $F_A, F_B, F_C, F_D$ are multi-variable generalisations of the hypergeometric function ${}_2F_1$ \cite{Lauricella1893Sulle}. In this paper we are mostly concerned with the fourth Lauricella function:
	\begin{align*}
	F_D^{(m)}	&\left[a, b_1, \ldots, b_m; c; x_1, \ldots, x_m \right]\\
	&= \sum_{k_1, \ldots, k_m = 0}^{\infty} \frac{\pochhammer{a}{\sum_j k_j} \pochhammer{b_1}{k_1} \ldots \pochhammer{b_m}{k_m}}{\pochhammer{c}{\sum_j k_j}} \frac{x_1^{k_1}}{k_1!} \ldots  \frac{x_m^{k_m}}{k_m!}.
	\end{align*}
	If $a = -n$ is a negative integer, then the sum becomes $\ds \sum_{k_1 + \ldots k_m \leq n}$ and the function is a polynomial, so it converges for all values of the $x_i$. The function is not defined if $c$ is a negative integer. 
	
	\subsection{Kamp\'{e} de Feriet function}
	
	The Kamp\'{e} de Feriet function \cite{Feriet1937La} generalises the hypergeometric function to two variables in a different way to Lauricella functions:
	\begin{align*}
	{}^{p+q}F_{r+s}	&\left[\begin{array}{l l}
	a_1, \ldots, a_p: b_1, b_1'; \ldots; b_q, b_q';\\
	c_1, \ldots, c_r: d_1, d_1'; \ldots; d_s, d_s';\end{array}  x_1, x_2 \right]\\
	&= \sum_{k_1 = 0}^{\infty} \sum_{k_2 = 0}^{\infty}  \frac{\pochhammer{a_1}{k_1 + k_2} \ldots \pochhammer{a_p}{k_1 + k_2} \pochhammer{b_1}{k_1} \pochhammer{b_1'}{k_1'} \ldots \pochhammer{b_q}{k_1} \pochhammer{b_q'}{k_2'}}{\pochhammer{c_1}{k_1 + k_2} \ldots \pochhammer{c_r}{k_1 + k_2} \pochhammer{d_1}{k_1} \pochhammer{d_1'}{k_1'} \ldots \pochhammer{d_s}{k_1} \pochhammer{d_s'}{k_2'}} \frac{x_1^{k_1}}{k_1!}  \frac{x_2^{k_2}}{k_2!}.
	\end{align*}
	
	\subsection{Lauricella--Kamp\'{e} de F\'{e}riet functions}
		It seems natural to extend hypergeometric functions in both directions to create the following function of $m$ variables: 
	\begin{align*}
	&{}^{p+q}F^{(m)}_{r+s}	\left[\begin{array}{l l l}
	\textbf{a} : B;\\
	\textbf{c} : D;\end{array}  x_1, \ldots, x_m \right]\\
	&={}^{p+q}F^{(m)}_{r+s}	\left[\begin{array}{l l l}
	a_1, \ldots, a_p &: b_{11}, \ldots, b_{1m} ; \ldots; b_{q1}, \ldots, b_{qm};\\
	c_1, \ldots, c_r &: d_{11}, \ldots, d_{1m} ; \ldots; d_{s1}, \ldots, d_{sm};\end{array}  x_1, \ldots, x_m \right]\\
	&= \sum_{k_1, \ldots, k_m = 0}^{\infty}  \frac{\ds \prod_{i=1}^p \pochhammer{a_i}{\sum_j k_j} \prod_{i=1}^q \prod_{j=1}^m \pochhammer{b_{ij}}{k_j}}{\ds \prod_{i=1}^r \pochhammer{c_i}{\sum_j k_j} \prod_{i=1}^s \prod_{j=1}^m \pochhammer{c_{ij}}{k_j}}
	\frac{x_1^{k_1}}{k_1!} \ldots \frac{x_m^{k_m}}{k_m!}.
	\end{align*}
	We could not find a name for this function in the literature, but a similar function appears in \cite{exton1976multiple}. It can be expressed as a (very) special case of the Srivivasta--Daoust function. In this paper, we only need the special case 
	\begin{align*}
	F_n &= {}^{1+1}F_{0+0} \left( \begin{array}{l l l}
	1 - n	&; n p_1, \ldots, n p_m	&; \\
	&;						&;
	\end{array} \frac{1}{n t_1}, \ldots, \frac{1}{n t_m} \right)\\
	&= \sum_{k_1, \ldots, k_m =0}^\infty \frac{(1-n)_{k_1 + \ldots + k_m} (np_1)_{k_1} \ldots (np_m)_{k_m}}{k_1! \ldots k_m! (nt_1)^{k_1} \ldots (nt_m)^{k_m}}.
	\end{align*}
	This is always a polynomial of the terms $p_i$ and $\frac{1}{t_i}$ of order $n$.
	
	\section{The generalized Chu--Vandermonde identity}
	
	The generalized Chu-Vandermonde identity is the latest in a line of combinatorics identities dating back to 1303 \cite{askey1975orthogonal}. This generalisation is due to Favaro \cite{favaro2012generalized} and another proof can be found in \cite{cerquetti2010simple}.
	
	\begin{lem}[Generalized Chu-Vandermonde identity]
		Let $n, k, r$ be positive integers. Let $q_1, \ldots, q_r \in \mathbb{C}$ and $w_1, \ldots, w_r  \in \mathbb{C}$. Let $i \in \{1, \ldots, r\}$. Then if $q_i \neq -(k-1), \ldots, -1, 0$ and $w_i \neq 0$, then 	
		
		\begin{align*}
		S	&= \sum_{\Sigma_j k_j = k} \binom{k}{k_1, \ldots, k_r} \prod_{j=1}^r (q_j)_{k_j} w_j^{k_j}\\
			&= w_i^k \left(\sum_j q_j \right)_k F_D^{(r-1)} \left(-k, \underbrace{q_1,\ldots, q_r}_{\neq i}, \sum_j q_j ; \underbrace{1 - \frac{w_1}{w_i}, \ldots, 1 - \frac{w_r}{w_i}}_{\neq i}\right).
		\end{align*}
		Note: the $\neq i$ under the brace indicates that the term with $i$ is removed. 
	\end{lem}
	When using this identity for a particular $i$, we will say the identity is \emph{centered on} $i$. In \cite{favaro2012generalized}, all the $q_j$ and $w_j$ are assumed to be positive, but this condition is not necessary. Indeed, the only condition is on $q_i$ and $w_i$.
	
	\begin{proof}
		We follow the proof in \cite{cerquetti2010simple}.  We use four short identities, which hold for all nonnegative integers $k_1 + \ldots + k_r = k$. The first follows from the definition of the Pochhammer symbol.
		\begin{equation} \label{short identity 1}
		\frac{k!}{k_1! \ldots k_r!} = 	\frac{k_i!}{k_1! \ldots k_r!} (-1)^{k - k_i} \, (-k)_{k - k_i}
		\end{equation}
		For all $q \neq -(k-1), \ldots, -1, 0$, we have
		\begin{equation} \label{short identity 2}
		(q)_{k_i} = \frac{(-1)^{k - k_i} \, (q)_k }{(1 - q - k)_{k - k_i}}
		\end{equation}
		For $c \neq 0$, the fourth Lauricella function satisfies (see \cite[page 216]{exton1976multiple})
		\begin{equation}\label{Lauricella reflection}
		\begin{split}
		F_D^{(m)}	&[-k, b_1, \ldots, b_m; c; x_1, \ldots x_m]\\
		&= \frac{(c- \sum_j b_j)_k}{(c)_k} F_D^{(m)}\left[-k, b_1, \ldots, b_r; 1 + \sum_j b_j - k - c; 1 - x_1, \ldots, 1 - x_m \right].
		\end{split} 	
		\end{equation}
		For all $Q \in \mathbb{C}$ and all $q \neq -(k-1), \ldots, -1, 0$:
		\begin{equation} \label{short identity 3}
		(q)_k \frac{(1 - k - Q)_k}{(1 - k - q)_k} = (Q)_k.
		\end{equation}		
		Let 
		$$S = \sum_{\sum_j k_j = k} \frac{k!}{k_1! \ldots k_r!} \prod_{j=1}^m (q_j)_{k_j} w_j^{k_j}.$$
		Using \cref{short identity 1}, we can extract the terms with an $i$:
		\begin{align*}
		S	&= \sum_{\sum_j k_j = k} \frac{(-1)^{k - k_i} (-k)_{k - k_i} k_i!}{k_1! \ldots k_r!} \prod_{j=1}^m (q_j)_{k_j} w_j^{k_j}.\\
		&= w_i^k \sum_{\sum_j k_j = k} \frac{(-1)^{k - k_i} (-k)_{k - k_i} k_i!}{k_1! \ldots k_r!} (q_i)_{k_i} \prod_{j \neq i} (q_j)_{k_j} \left(\frac{w_j}{w_r} \right)^{k_j}.
		\end{align*}
		By \cref{short identity 2}, 
		\begin{align*}
		S	&= w_i^k \sum_{\sum_j k_j = k} \frac{(-1)^{k - k_i} (-k)_{k - k_i} k_i}{k_1! \ldots k_r!} \frac{(-1)^{k - k_i} (q_i)_k }{(1 - q_i - k)_{k - k_i}} \prod_{j \neq i} (q_j)_{k_j} \left(\frac{w_j}{w_i} \right)^{k_j}\\
		&= w_i^k \sum_{\sum_j k_j = k} \frac{(-k)_{k - k_i} (q_i)_k }{(1 - q_i - k)_{k - k_i}} \prod_{j=1}^{r-1} \frac{(q_j)_{k_j}}{k_j!} \left(\frac{w_j}{w_i} \right)^{k_j} \\
		&= w_i^k (q_i)_k \sum_{\sum_j^{r-1} k_j \leq k} \frac{(-k)_{\sum_j^{r-1} k_j}}{(1 - q_i - k)_{\sum_j^{r-1} k_j}} \prod_{j=1}^{r-1} \frac{(q_j)_{k_j}}{k_j!} \left(\frac{w_j}{w_i} \right)^{k_j} \\
		&= w_i^k (q_i)_k F_D^{(r-1)}\left[-k, \underbrace{q_1, \ldots, q_r}_{\neq i}; 1 - q_i  - k; \underbrace{\frac{w_1}{w_i}, \ldots, \frac{w_r}{w_i}}_{\neq i} \right].
		\end{align*}
		Let $c = 1 - q_i - k$ and note that $1 + \sum_{j \neq i} q_i - n - c = \sum_j q_j$. Then by \cref{Lauricella reflection},
		\begin{align*}
		S	&= w_r^k (q_r)_k \frac{(1 - k - \sum_j q_j)_k}{(1 - q_r  - k )_k}\\
		&\times F_D^{(r-1)}\left[-k, \underbrace{q_1, \ldots, q_r}_{\neq i}; \sum_j q_j; \underbrace{1 - \frac{w_1}{w_i}, \ldots, 1 - \frac{w_r}{w_i}}_{\neq i} \right].
		\end{align*}
		Finally, with \cref{short identity 3}, we have
		\begin{align*}
		S	&= w_r^k \left(\sum_j q_j \right)_k F_D^{(r-1)}\left[-k, \underbrace{q_1, \ldots, q_r}_{\neq i}; \sum_j q_j; \underbrace{1 - \frac{w_1}{w_i}, \ldots, 1 - \frac{w_r}{w_i}}_{\neq i} \right].
		\end{align*}

	\end{proof}
	
	\section{Inverses using the limit definition of the exponential function}
	
	We use the limit definition of the exponential function 
	$$e^x = \lim_{n \rightarrow \infty} \left(1 + \frac{x}{n} \right)^n.$$
	This limit is locally uniform. 
	
	\begin{lem} 
		If $f_n$ is a sequence of biholomorphic functions converging to $f$ locally uniformly, then $f$ has an inverse $h$ and $f_n^{-1}$ converges to $h$ locally uniformly.	
	\end{lem}
	\begin{proof}
		Let $y = f(x)$. By the inverse function theorem and limit properties, 
		\begin{align*}
		\left(f^{-1}\right)'(y) &= \frac{1}{f'(x)}	= \lim_{n \rightarrow \infty} \frac{1}{f_n'(x)}\\
		&= \lim_{n \rightarrow \infty} \left(f_n^{-1} \right)'(f_nf^{-1} y)\\
		&= \lim_{n \rightarrow \infty} \left(f_n^{-1} \right)'(\lim_{n \rightarrow \infty} f_nf^{-1} y)\\
		&= \lim_{n \rightarrow \infty} \left(f_n^{-1} \right)'(y),\\
		f^{-1}(y) - f^{-1}(0)&= \lim_{n \rightarrow \infty} f_n^{-1}(y) - f_n^{-1}(0),\\
		f^{-1}(y)	&= \lim_{n \rightarrow \infty} f_n^{-1}(y).
		\end{align*}
		See \cite{Convergence} for another proof.
	\end{proof}
	We can write $f(z) = z R(z) e^z$. The derivative is $f'(z) = R(z)e^z + z R'(z) e^z + z R(z) e^z$. Note that $R(0) \neq 0$, so $f'(0) = R(0)e^0 \neq 0$. So $f$ is biholomorphic on some open neighborhood $V$ of $0$. Similarly, $f_n(z) = z R(z) \left(1 + \frac{z}{n} \right)^n$ is biholomorphic on some open neighborhood $V_n$ of $0$. Let $x \in V$ and $y = f(x)$.	The sequence $\lim_{n \rightarrow \infty} \left(1 + \frac{x}{n} \right)^n$ converges locally uniformly to $e^x$. So the sequence 
	$$f_\tau(z) =  z (z - t_1)^{p_1} \ldots (z - t_m)^{p_m} \left(1 + \frac{z}{\tau}\right)^\tau$$
	converges locally uniformly to $f(z)$. So we have
	
	$$\lim_{n \rightarrow \infty} f^{-1}_n = f^{-1}.$$

	\section{Main theorem} \label{Main theorem section}
	
	\begin{theorem}
		Let $f(z) = z (z - t_1)^{p_1} \ldots (z - t_m)^{p_m} e^z$. Then the Taylor series for the inverse of $f$ around $0$ is 
		\begin{equation} 
		f^{-1}(z) = \sum_{n=1}^\infty \frac{(-n)^{n-1}}{n!} (-t_1)^{-n p_1} \ldots (-t_m)^{-n p_m} F_n z^n,
		\end{equation}
		where 
		$$F_n = {}^{1+1}F_{0+0} \left( \begin{array}{l l l}
		1 - n	&; n p_1, \ldots, n p_m	&; \\
		&;						&;
		\end{array} \frac{1}{n t_1}, \ldots, \frac{1}{n t_m} \right)$$
		is a generalized Kamp\'{e} de F\'{e}riet function.
	\end{theorem}
	
	\begin{proof}
		The idea of the proof is to use the limit definition of $e^z$ to represent $f(z)$ as a product of powers of linear terms in $z$. This simplifies the application of the Lagrange inversion theorem, so that we can apply the generalized Chu-Vandermonde identity centered on the exponential term.
		
		Let $f_\tau(z) = z (z - t_1)^{p_1} \ldots (z - t_m)^{p_m} \left(1 + \frac{z}{\tau}\right)^\tau$, so that $\ds \lim_{\tau \rightarrow \infty} f_\tau(z) = f(z)$. Then we use the Lagrange inversion theorem \cite{Abramowitz1964Handbook} on $f_\tau(z)$.	
		\begin{align*}
		f_\tau^{-1}(z)	&=  \sum_{n = 1}^\infty \frac{z^n}{n!} \lim_{w \rightarrow 0} \frac{d^{n-1}}{dw^{n-1}}\left( \frac{w}{f(w)} \right)^n\\
		&= \sum_{n = 1}^\infty \frac{z^n}{n!} \lim_{w \rightarrow0} \frac{d^{n-1}}{dw^{n-1}}\left( \frac{\tau^\tau}{(w - t_1)^{p_1}(w - t_2)^{p_2} \ldots (w - t_m)^{p_m} (w + \tau)^{\tau}} \right)^n\\
		&= \sum_{n = 1}^\infty \frac{z^n}{n!} \tau^{\tau n} \lim_{w \rightarrow 0} \frac{d^{n-1}}{dw^{n-1}}\left( (w - t_1)^{-n p_1}\ldots (w - t_m)^{-n p_m} (w + \tau)^{-n \tau} \right).	
		\end{align*}
		To match the indices, let $p_0 = \tau$ and $t_0 = - \tau$. Using the generalized product rule, we can expand the derivative into a sum over vectors $\textbf{k} = (k_0, \ldots k_m)$ such that $k_0 + \ldots + k_m = n-1$. 
		
		\begin{align*}
		f_\tau^{-1}(z)	&= \sum_{n = 1}^\infty \frac{z^n}{n!} \tau^{\tau n} \lim_{w \rightarrow 0} \sum_{\sum_j k_j = n-1} \binom{n-1}{\textbf{k}} \prod_{i=0}^m (-n p_i)^{\underline{k_i}} (w - t_i)^{-n p_i - k_i}\\
		&= \sum_{n = 1}^\infty \frac{z^n}{n!} \prod_{i=0}^m (- t_i)^{-n p_i}\cdot \tau^{\tau n}  \sum_{\textbf{k}} \binom{n-1}{\textbf{k}}\prod_{i=0}^m (-n p_i)^{\underline{k_i}} (- t_i)^{-k_i}\\
		&= \sum_{n = 1}^\infty \frac{z^n}{n!} \prod_{i=1}^m (- t_i)^{-n p_i}\cdot \sum_{\textbf{k}} \binom{n-1}{\textbf{k}}\prod_{i=0}^m (-n p_i)^{\underline{k_i}} (- t_i)^{-k_i}\\	
		\end{align*}
		where $x^{\underline{a}}$ is the falling factorial $x (x-1) \ldots (x - a + 1)$. This can be converted to a rising factorial or Pochhammer symbol by the formula $x^{\underline{a}} = (-1)^a (-x)_a$. 
		\begin{align*}
		f_\tau^{-1}(z)	&= \sum_{n = 1}^\infty \frac{z^n}{n!} \prod_{i=1}^m (- t_i)^{-n p_i} \sum_{\textbf{k}} \binom{n-1}{\textbf{k}} \prod_{i=0}^m (-1)^{k_i}(n p_i)_{k_i} (- t_i)^{-k_i}\\
		&= \sum_{n = 1}^\infty \frac{z^n}{n!} \prod_{i=1}^m (- t_i)^{-n p_i} \sum_{\textbf{k}} \binom{n-1}{\textbf{k}}\prod_{i=0}^m (n p_i)_{k_i} \left(\frac{1}{t_i} \right)^{k_i}.
		\end{align*}
		Next we use the Chu-Vandermonde identity, centered on $0$. Note that $n p_0 = n \tau > 0$ and $\frac{-1}{t_0} = \frac{1}{\tau} > 0$, so the identity applies. Let $P = \sum_{i=1}^m p_i$. We have
		\begin{align*}
		&\sum_{\textbf{k}} \binom{n-1}{\textbf{k}} \prod_{i=1}^m (n p_i)_{k_i} \left(\frac{1}{t_i} \right)^{k_i} \\
		&= \left(\frac{1}{t_0} \right)^{n-1} \left( n P + n p_0\right)_{n-1} F_D^{(m)} \left(\begin{array}{l l}
		1-n;  n p_1,\ldots, n p_m\\
		nP + n p_0\end{array}; 1 -\frac{t_0}{t_1}, \ldots, 1 - \frac{t_0}{t_m}\right)\\
		&= \left(\frac{-1}{\tau} \right)^{n-1} \left( n P + n \tau \right)_{n-1} F_D^{(m)} \left(\begin{array}{l l}
		1-n;  n p_1,\ldots, n p_m\\
		nP + n \tau\end{array}; 1 + \frac{\tau}{t_1}, \ldots, 1 + \frac{\tau}{t_m}\right).
		\end{align*}
		Next we start taking the limit as $\tau \rightarrow \infty$. We have 
		$$\lim_{\tau \rightarrow \infty} \frac{(n P + n \tau)_{n-1}}{(-\tau)^{n-1}} = (-n)^{n-1}.$$
		Then we take the limit of the Lauricella function. We have 
		\begin{align*}
		F_D^{(m)} &= \sum_{k_1, \ldots, k_m =0}^\infty \frac{(1-n)_{k_1 + \ldots + k_m} (np_1)_{k_1} \ldots (np_m)_{k_m}}{(n P + n \tau)_{k_1 + \ldots + k_m}} \prod_{i=1}^m \frac{1}{k_i!}\left(1 + \frac{\tau}{t_i} \right)^{k_i}.
		\end{align*}
		Note that the highest power of $\tau$ in the numerator and denominator is $\tau^{k_1 + \ldots + k_m}$, so we have the limit
		\begin{align*}
		\lim_{\tau \rightarrow \infty}\frac{\prod_{i=1}^m \left(1 + \frac{\tau}{t_i} \right)^{k_i}}{(n P + n \tau)_{k_1 + \ldots + k_m}} &= \frac{1}{(n t_1)^{k_1} \ldots (n t_m)^{k_m}}.
		\end{align*}
		So 
		\begin{align*}
		\lim_{\tau \rightarrow \infty} F_D^{(m)} &= \sum_{k_1, \ldots, k_m =0}^\infty \frac{(1-n)_{k_1 + \ldots + k_m} (np_1)_{k_1} \ldots (np_m)_{k_m}}{k_1! \ldots k_m! (nt_1)^{k_1} \ldots (nt_m)^{k_m}}\\
		&= {}^{1+1}F_{0+0} \left( \begin{array}{l l l}
		1 - n	&; n p_1, \ldots, n p_m	&; \\
		&;						&;
		\end{array} \frac{1}{n t_1}, \ldots, \frac{1}{n t_m} \right).
		\end{align*}
		So we have
		\begin{align*}
		f^{-1}(z)	&= \sum_{n = 1}^\infty \frac{(-n)^{n-1}}{n!} \prod_{i=1}^m (- t_i)^{-n p_i} F_n z^n.
		\end{align*}
		
	\end{proof}
		
\section{Radius of convergence} \label{ROC section}


By applying the root test for convergence, the radius of convergence for the series in \cref{series} is given by 
\begin{align*}
\frac{1}{R} &= \limsup_{n \rightarrow \infty} \left| \frac{(-n)^{n-1}}{n!} (-t_1)^{-n p_1} \ldots (-t_m)^{-n p_m} F_n \right|^{1/n}\\
&= e |t_1|^{-p_1} \ldots |t_m|^{- p_m} \limsup_{n \rightarrow \infty} \left|  F_n \right|^{1/n}.
\end{align*}
So we are looking for 
$$\limsup_{n \rightarrow \infty} \left|\sum_{k_1, \ldots, k_m =0}^\infty \frac{(1-n)_{k_1 + \ldots + k_m} (np_1)_{k_1} \ldots (np_m)_{k_m}}{k_1! \ldots k_m! (nt_1)^{k_1} \ldots (nt_m)^{k_m}} \right|^{1/n}.$$
Denote the coefficients in the sum by
$$a_n(\textbf{k}) = \frac{(1-n)_{k_1 + \ldots + k_m} (np_1)_{k_1} \ldots (np_m)_{k_m}}{k_1! \ldots k_m! (nt_1)^{k_1} \ldots (nt_m)^{k_m}},$$
then let $\lambda = (\lambda_1, \ldots, \lambda_m)$, where $\lambda_i = \frac{k_i}{n}$. The asymptotic behaviour of $a_n(\lambda n)$ for large $n$ allows us to examine the asymptotic behaviour of $F_n$. Note that when $k_1 + \ldots + k_m \geq n-1$ we have $a_n(\textbf{k}) = 0$. So 
$$\limsup_{n \rightarrow \infty} F_n = \limsup_{n \rightarrow \infty} n \int_{\Delta} a_n(\lambda n) \, d\lambda,$$
where $\Delta$ is the simplex 
$$\Delta = \left\{\lambda = (\lambda_1, \ldots, \lambda_m): \lambda_i \geq 0 \text{ and } \sum_j \lambda_j \leq 1 \right\}.$$

\begin{lem}
	For any $p \in \mathbb{C}$, positive integers $n$ and $\lambda \in [0, 1]$, we have
	\begin{align*} 
	\frac{(n p)_{n \lambda}}{(n \lambda)!} \sim \sqrt{\frac{p}{2 \pi n \lambda (\lambda+p)}} \left( \frac{(p + \lambda)^{p + \lambda}}{p^p \lambda^\lambda} \right)^n \psi(n, \lambda, p), 
	\end{align*}
	where $\psi(n, \lambda p) = 2 i \sin (n p \pi) e^{i \pi p n}$ if $p < 0 < p + \lambda$ and $1$ otherwise.
	
	\begin{proof}
		First we expand the Pochhammer symbols into Gamma functions. We have 
		$$(n p)_k = \left\{ \begin{array}{l l}
		\ds \frac{\Gamma(n p + k)}{\Gamma(n p)}	& \text{ if $n p  > 0$},\\
		\ds \Gamma(n p + k)\Gamma(1 - n p) \frac{\sin(n p \pi)}{\pi}	& \text{ if $n p < 0 < n p + k$,}\\
		\ds \frac{(-1)^k \Gamma (1- n p)}{\Gamma (1-k-n p)}	& \text{ if $n p + k < 0$}.
		\end{array}\right.$$
		We use Stirling's approximation $\Gamma(z) \sim \sqrt{\frac{2\pi}{z}} \left( \frac{z}{e} \right)^{z}$ for large $|z|$ and $\arg z < \pi$.
		$\Gamma(z+1) \sim \sqrt{2\pi z} \left( \frac{z}{e} \right)^{z}$ for large $|z|$ and $\arg z < \pi$.
		If $\arg p \neq \pi$, then 
		\begin{align*}
		\frac{(n p)_{n \lambda}}{(n \lambda)!}	&\sim \frac{\Gamma(n p + n \lambda)}{\Gamma(n p) \Gamma(n \lambda + 1)}\\
		&= \frac{\sqrt{\frac{2\pi}{n (p + \lambda)}} \left( \frac{n (p + \lambda)}{e} \right)^{n (p + \lambda)}}{\sqrt{\frac{2\pi}{n p}} \left( \frac{n p}{e} \right)^{n p} \sqrt{2\pi n \lambda} \left( \frac{n \lambda}{e} \right)^{n \lambda}}\\
		&= \sqrt{\frac{p}{2\pi n \lambda (p + \lambda)}}\left( \frac{(p + \lambda)^{p + \lambda}}{p^p \lambda^\lambda} \right)^n.
		\end{align*}
		If $p < 0 < p + \lambda$, then 
		\begin{align*}
		\frac{(n p)_{n \lambda}}{(n \lambda)!}	&\sim \frac{\Gamma(n p + n \lambda)\Gamma(1 - n p)}{\Gamma(n \lambda + 1)} \frac{\sin(n p \pi)}{\pi}\\
		&= \frac{\sqrt{\frac{2\pi}{n (p + \lambda)}} \left( \frac{n (p + \lambda)}{e} \right)^{n (p + \lambda)}\sqrt{-2\pi n p} \left( \frac{- n p}{e} \right)^{-n p}}{\sqrt{2\pi n \lambda} \left( \frac{n \lambda}{e} \right)^{n \lambda}} \frac{\sin(n p \pi)}{\pi}\\
		&= 2 i \sin(n p \pi) e^{i \pi p n} \sqrt{\frac{p}{2 \pi n \lambda (p + \lambda)}} \left( \frac{(p + \lambda)^{p + \lambda} }{p^p \lambda^\lambda} \right)^{n}.
		\end{align*}
		If $p + \lambda < 0$, then 
		\begin{align*}
		\frac{(n p)_{n \lambda}}{(n \lambda)!}	&\sim \frac{(-1)^{n \lambda} \Gamma (1- n p)}{\Gamma (1-n (p + \lambda)) \Gamma(n \lambda + 1)}\\
		&= \frac{(-1)^{n \lambda} \sqrt{- 2\pi n p} \left( \frac{- n p}{e} \right)^{-n p}}{\sqrt{- 2\pi n (p + \lambda)} \left( \frac{-n (p+\lambda)}{e} \right)^{-n (p + \lambda)} \sqrt{2\pi n \lambda} \left( \frac{n \lambda}{e} \right)^{n \lambda}}\\
		&= (-1)^{n \lambda} \sqrt{\frac{p}{2\pi n \lambda(p + \lambda)}} \left( \frac{ (-p-\lambda)^{p + \lambda}}{(-p)^{p}\lambda^{\lambda}} \right)^n\\
		&= \sqrt{\frac{p}{2\pi n \lambda(p + \lambda)}} \left( \frac{(p + \lambda)^{p + \lambda}}{p^p \lambda^\lambda} \right)^n.
		\end{align*}	
		Combining these, we have 		
		$$\frac{(n p)_{n \lambda}}{(n \lambda)!} \sim \psi(n, \lambda, p) \sqrt{\frac{p}{2\pi n \lambda(p + \lambda)}} \left( \frac{(p + \lambda)^{p + \lambda}}{p^p \lambda^\lambda} \right)^n,$$
		where
		$$\psi(n, \lambda, p) = \left\{ \begin{array}{l l}
		2 i \sin (n p \pi) e^{i \pi p n}	&	-\lambda < p < 0,\\
		1		&	\text{otherwise}.
		\end{array}\right.$$
	\end{proof}
\end{lem}

\begin{lem}
	For integers $n > 0$ and $k = n \lambda$ with $0< \lambda < 1$, 
	\begin{align*}
	(1-n)_k	&= \frac{(-1)^{n \lambda} \Gamma(n)}{\Gamma(n(1 - \lambda))}\\
	&\sim \frac{(-1)^{n \lambda} \sqrt{\frac{2 \pi }{n}} \left(\frac{n}{e}\right)^n}{\sqrt{\frac{2 \pi }{(1-\lambda ) n}} \left(\frac{(1-\lambda ) n}{e}\right)^{(1-\lambda ) n}}\\
	&=(-1)^{n \lambda} \sqrt{1 - \lambda} \left(\frac{n^{\lambda}}{\left(1-\lambda \right)^{1-\lambda} e^{\lambda}} \right)^n.
	\end{align*}
\end{lem}

	Let $\Sigma_\lambda = \sum_i \lambda_i$. Let $\psi_i = \psi(n, \lambda_i, p_i)$. The $t_i$ can be complex, so let $t_i = |t_i| e^{i \theta_i}$. We have
	\begin{align*}
	a_n(\textbf{k})			&= (1-n)_{k_1 + \ldots + k_m}  \prod_{i=1}^r \frac{(np_i)_{k_i}}{k_i! (n t_i)^{k_i}},\\
	a_n(n \lambda)	&= \frac{(-1)^{n \Sigma_\lambda} \Gamma(n)}{\Gamma(n(1 - \Sigma_\lambda))}  \prod_{i=1}^r \frac{(np_i)_{n \lambda_i}}{(n \lambda_i)! (n t_i)^{n \lambda_i}}\\
	&\sim (-1)^{n \Sigma_\lambda} \sqrt{1 - \Sigma_\lambda} \left(\frac{n^{\Sigma_\lambda}}{\left(1-\Sigma_\lambda \right)^{1-\Sigma_\lambda} e^{\Sigma_\lambda}} \right)^n \prod_{i=1}^r \psi_i \sqrt{\frac{p_i}{2\pi n \lambda_i (p_i + \lambda_i)}} \left(\frac{(p_i + \lambda_i)^{p_i + \lambda_i}}{p_i^{p_i} (n t_i \lambda_i)^{\lambda_i}} \right)^n\\
	&= \prod_{j=1}^n \psi_j  \left( \sqrt{1 - \Sigma_\lambda} \prod_{j=1}^n \sqrt{\frac{p_i}{2\pi n \lambda_i (p_i + \lambda_i)}} \right) \left(\frac{1}{\left(1-\Sigma_\lambda \right)^{1-\Sigma_\lambda}} \ds \prod_{j=1}^r  \frac{(p_i + \lambda_i)^{p_i + \lambda_i}}{p_i^{p_i} (-e t_i \lambda_i)^{\lambda_i}} \right)^n\\
	&= \prod_{j=1}^n \frac{\psi_j}{\sqrt{n}}  \cdot f(\lambda) e^{n g(\lambda)}, 
	\end{align*}	
	where 
	$$f(\lambda) = \sqrt{1 - \Sigma_\lambda} \prod_{j=1}^n \sqrt{\frac{p_i}{2\pi \lambda_i (p_i + \lambda_i)}}$$
	and
	$$g(\lambda) = \log \frac{1}{\left(1-\Sigma_\lambda \right)^{1-\Sigma_\lambda}} \ds \prod_{j=1}^r  \frac{(p_i + \lambda_i)^{p_i + \lambda_i}}{p_i^{p_i} (-e t_i \lambda_i)^{\lambda_i}}.$$
So for large $n$, the sum becomes a multiple integral over the simplex:
$$\sum_{\textbf{k}} a_n(\textbf{k}) \sim n^{1 - \frac{m}{2}} \prod_{j=1}^n \psi_j \cdot \int_\Delta  f(\lambda) e^{n g(\lambda)} d\lambda.$$
If $p_j \in \mathbb{Q} \cap [-\lambda_j, 0]$ then $\psi_j = 0$ whenever $n p_j$ is an integer. We can choose arbitrarily large $n$ such that $n p_j$ is not an integer. So we have 
\begin{align*}
\limsup_{n \rightarrow \infty} \left|\sum_{\textbf{k}} a_n(\textbf{k}) \right|^{1/n} &= \limsup_{n \rightarrow \infty} \left|n^{1 - \frac{m}{2}} \prod_{j=1}^n \psi_j \cdot \int_\Delta  f(\lambda) e^{n g(\lambda)} d\lambda \right|^{1/n}\\
&= \limsup_{n \rightarrow \infty} \left|\int_\Delta f(\lambda) e^{n g(\lambda)} d\lambda \right|^{1/n}\\
\end{align*}

\subsection{Saddle point method}

To evaluate this limit we can use the \enquote{saddle point method}, following \cite{fedoryuk1987asymptotic, Wong2001Asymptotic}. We deform the simplex $\Delta$ into a \enquote{steepest descent} set $\Omega$ with the same boundary as $\Delta$ and containing one or more zeros $\phi_i$ of $g'(\lambda)$. The set $\Omega$ can be chosen in such a way that the imaginary part of $g(z)$ on $\Omega$ is constant. Then we have
\begin{align*}
\int_\Delta f(x) e^{n g(x)} dx &= \int_\Omega f(x) e^{n g(x)} dx\\
&\sim \left(\frac{2\pi}{n} \right)^{\frac{m}{2}} \sum_k e^{n g(\phi_k)} \det\left(- S''(\phi_k)\right)^{-\frac{1}{2}} f(\phi_k).
\end{align*}
Typically, one expects that of the $\phi_i$ will dominate the others, and the limit becomes
\begin{align*}
\limsup_{n \rightarrow \infty} \left| \int_0^1 f(x) e^{n g(x)} dx \right|^{1/n} &= e^{\Re (g(\phi_k))},
\end{align*}
for some $\phi_k$. It remains to find the saddle points of $g$. We have:
\begin{align} \label{derivative of g}
\frac{\partial g}{\partial \lambda_i} &= e^{g(\lambda)} \left( \log(1 - \sum \lambda_j) + \log(p_i + \lambda_i) - \log(- t_i \lambda_i) \right) = 0.
\end{align}
If the derivatives are zero, then
\begin{equation} \label{polynomial system}
(1 - \sum \lambda_j)(p_i + \lambda_i) + t_i \lambda_i = 0, 
\end{equation} 
for all $i$. However, if we substitute the solutions to \cref{polynomial system} into \cref{derivative of g} we may get an integer multiple of $2\pi i$ instead of $0$, due to the branch cuts of the complex logarithm. For example, if $m = 1, p = -1 + i$, and $t = -1 + i$ then the solutions are $-i$ and $1 + i$. But $-i$ is not a root of \cref{derivative of g}. From numerical experiments using Mathematica, it seems that the equation 
$$\limsup_{n \rightarrow \infty} \left| F_n \right|^{1/n} = e^{\Re (g(\phi))}$$
for some solution $\phi$, holds whenever both solutions to \cref{polynomial system} are also solutions to \cref{derivative of g}. Otherwise, it may not hold for any of the solutions. Note that adding a integer (say $l_i$) multiple of $2\pi i \lambda_i$ to $g$ makes no difference to the sum $F_n$, because $\lambda_i n$ is always an integer. That is, 
$$e^{ng(\lambda)} = e^{n g(\lambda) + 2\pi i \lambda_i n l_i }.$$
However this addition does make a difference to the integral, because the integral doesn't assume $n \lambda$ is an integer. If we include this extra term, then to find the saddle point we need to solve 
\begin{equation} \label{extra complex bit}
(1 - \sum \lambda_j)(p_i + \lambda_i) + t_i \lambda_i = - 2\pi i l_i. 
\end{equation} 
The solutions to this must also be solutions to \cref{polynomial system}, but this may recover missing solutions. We suspect that the $l_i$ can only take the values $-1, 0, 1$. Now we get 
\begin{equation} \label{Integral limit}
\limsup_{n \rightarrow \infty} \left| \int_{\Delta} f(\lambda) e^{n g(\lambda) + 2\pi i n \sum_j\lambda_j l_j} d \lambda \right|^{1/n} = e^{g(\phi) + \sum_j 2\pi i l_j \phi_j},
\end{equation} 
where $\phi$ is a solution to \cref{polynomial system}. Experiments with Mathematica suggest that this holds for many values of $p$ and $t$, including complex values. If \cref{Integral limit} is true, then \cref{Radius conjecture} follows by a simple calculation. A full proof of the conjecture may be difficult, due to the following complications:
\begin{enumerate}
	\item For some values of $p$ and $t$, the saddle points may be \enquote{degenerate}, i.e. $\det g''(\lambda) = 0$. These special cases can be handled with catastrophe theory. 
	\item The functions $f$ and $g$ have several discontinuities and branch cuts that require special care.
	\item The extra integers $l_i$ are ad hoc, and it's not clear why they should work.
\end{enumerate}

\bibliographystyle{amsplain}
\bibliography{Lambert}

\end{document}